\documentclass[a4paper,reqno,12pt]{amsart}

\usepackage[T1]{fontenc}
\usepackage{amsthm}
\usepackage{amsmath,amssymb}
\usepackage{url}
\usepackage[mathscr]{euscript}
\usepackage{newtxtext}
\usepackage{newtxmath}
\usepackage{xcolor}
\usepackage[hidelinks]{hyperref}

\usepackage{fullpage}
\usepackage{setspace}
\usepackage{mathtools}
\mathtoolsset{showonlyrefs}

\def\today{\ifcase\month\or
  January\or February\or March\or April\or May\or June\or
  July\or August\or September\or October\or November\or December\fi
  \space\number\day, \number\year}

\newtheorem{theorem}{Theorem}

\newtheorem{corollary}[theorem]{Corollary}

\newcommand{\D}{\mathscr{D}}
\newcommand{\E}{\mathbb{E}}
\newcommand{\N}{\mathscr{N}}
\renewcommand{\r}{\mathbb{R}}
\newcommand{\z}{\mathbb{Z}}
\newcommand{\1}{\boldsymbol{1}}
\newcommand{\p}{\varphi}
\newcommand{\maj}{\prec_{\text{maj}}}

\begin{document}

\title[A component-wise inequality for permutation matches]{A component-wise inequality for permutation matches}
\author[Gon\c{c}alves]{Felipe Gon\c{c}alves}
\date{\today}
\subjclass[2020]{60E15, 26D15, 05A20}
\keywords{majorization, Schur convexity, matching positions, multinomial distribution, rearrangements}
\address{IMPA - Instituto de Matemática Pura e Aplicada, Rio de Janeiro, 22460-320, Brazil.}
\email{goncalves@impa.br}

\allowdisplaybreaks

\begin{abstract}
Motivated by the recent paper [Sharp endpoint extension inequalities for the moment curve on finite fields II: an extremal property of the uniform distribution, arXiv:2609.29882], {which} proves sharp extension inequalities in finite fields via a two-point symmetrization argument, we prove here a more general component-wise inequality for permutation matches that implies theirs. 
\end{abstract}

\maketitle

\section{Introduction}

Our motivation comes from the extremal problem considered in \cite{BCF}. The authors show that the uniform distribution maximizes the expected number of distinct rearrangements of an independent sample. We recover this result from a more general inequality. The proof is inspired by the technique in the {recent paper} \cite{GMR} of the author.

Given two vectors $a,b\in\r^q$, we write $a\maj b$ if
\[
\sum_{j=1}^\ell a_j^\downarrow\leq\sum_{j=1}^\ell b_j^\downarrow
\]
for all $\ell=1,\ldots,q$, with equality when $\ell=q$, where $\downarrow$ denotes decreasing rearrangement. For {this notion} we refer to \cite{MOA}. For integers $d,q\geq 1$ define $\N_{d,q}:=\{n\in\z_{\geq0}^q:n_1+\cdots+n_q=d\}$. For $n \in\N_{d,q}$ let
\[
t(n)=(\underbrace{1,\ldots,1}_{n_1},\underbrace{2,\ldots,2}_{n_2},\ldots,\underbrace{q,\ldots,q}_{n_q})
\]
be the canonical word with multiplicity vector $n$. Define the number of matching positions and its vector {indexed by permutations} by
\begin{equation}\label{eq:matches}
D_n(\pi)=\#\{1\leq i\leq d:t_i(n)=t_{\pi(i)}(n)\}
\quad\text{and}\quad
\D_n=(D_n(\pi))_{\pi\in S_d}.
\end{equation}

The following is the main result of this paper.

\begin{theorem}\label{thm:main}
Let $n,m\in\N_{d,q}$ satisfy $n\maj m$. Then, component-wise, we have
\begin{equation*}%\label{eq:sorted}
\D_n^\downarrow\leq\D_m^\downarrow.
\end{equation*}
\end{theorem}

\section{Applications}
We now describe two direct applications. 

For a given function $\p:[0,1]\to \r$ and $n\in\N_{d,q}$ we define
\begin{equation}\label{eq:Aphi}
A_\p(n)=\frac1{d!}\sum_{\pi\in S_d}\p\left(\frac{D_n(\pi)}d\right).
\end{equation}

\begin{corollary}\label{cor:family}
Let $\theta$ be a probability distribution on $X=\{1,2,\ldots,q\}$ and  $\theta_0$ be the uniform distribution on $X$. Let $T=(T_1,\ldots,T_d)$ be an independent sample with common distribution $\theta$, where $d\geq2$. Write $n(T) \in \N_{d,q}$ for its multiplicity vector. Let $\p:[0,1]\to I$ and  $g:I\to\r$ be nonincreasing, where $I\subset \r$ is an interval. Then
\[
\E_{\theta_0}g(A_\p(n(T))) \leq \E_\theta g(A_\p(n(T))).
\]
If $g$ is strictly decreasing and $\p(1-2/d)>\p(1)$, equality holds if and only if $\theta=\theta_0$.
\end{corollary}

\begin{proof}
Write $\theta_j$ for the probability of letter $j$. For nonuniform $\theta$, relabel so that $\theta_1>\theta_2$, and let $\bar\theta$ replace these two probabilities by their average. For each vector $r=(r_3,\ldots,r_q)$ of nonnegative integers with $r_3+\cdots+r_q\leq d$, put $k=d-r_3-\cdots-r_q$ and define, for $0\leq j\leq k$,
\[
f_j=g(A_\p(j,k-j,r))
\quad\text{and}\quad
w_j=\theta_1^j\theta_2^{k-j}+\theta_1^{k-j}\theta_2^j.
\]
Split $\N_{d,q}$ into $n_1=0$ and $n_1>0$. The substitutions $(0,n_2)\mapsto(n_2,0)$ and $(n_1,n_2)\mapsto(n_1-1,n_2+1)$ reindex these parts as $n_2=0$ and $n_2>0$, respectively, and give
\begin{align*}
\E_\theta g(A_\p(n(T)))&=d!\sum_r\frac1{k!}\prod_{\ell=3}^q\frac{\theta_\ell^{r_\ell}}{r_\ell!}
\left[f_0\theta_2^k+\sum_{j=1}^k\binom kj f_j\theta_1^j\theta_2^{k-j}\right]\\
& =d!\sum_r\frac1{k!}\prod_{\ell=3}^q\frac{\theta_\ell^{r_\ell}}{r_\ell!}
\left[f_k\theta_2^k+\sum_{j=0}^{k-1}\binom{k}{j+1}f_{j+1}\theta_1^{j+1}\theta_2^{k-j-1}\right]\\
&=\frac{d!}{2}\sum_r\frac1{k!}\prod_{\ell=3}^q\frac{\theta_\ell^{r_\ell}}{r_\ell!}
\sum_{j=0}^k\binom kj f_jw_j.
\end{align*}
In the last equality, we reverse the reindexed sum, use $f_j=f_{k-j}$, and average with the original sum. Applying this formula to $\theta$ and $\bar\theta$ gives
\begin{align*}
\E_\theta g(A_\p(n(T)))-\E_{\bar\theta}g(A_\p(n(T)))
&=\frac{d!}{2}\sum_r\frac1{k!}\prod_{\ell=3}^q\frac{\theta_\ell^{r_\ell}}{r_\ell!}
\sum_{j=0}^k\binom kj f_j\bigl(w_j-2\bar\theta_1^k\bigr)\\
&=d!\sum_r\frac1{2^{k+2}k!}\prod_{\ell=3}^q\frac{\theta_\ell^{r_\ell}}{r_\ell!}
\sum_{i,j=0}^k\binom ki\binom kj(f_i-f_j)(w_i-w_j)\\ & \geq0.
\end{align*}
By symmetry, the last sum can be restricted to $0\leq i,j\leq k/2$, counting each index below $k/2$ twice. On this range, Theorem \ref{thm:main} and the monotonicity of $\p$ and $g$ make $f_j$ nonincreasing, whereas $2\theta_1\theta_2\leq\theta_1^2+\theta_2^2$ makes $w_j$ nonincreasing. This proves the last inequality. Successive averaging and continuity give the desired comparison. If $g$ is strictly decreasing and $\p(1-2/d)>\p(1)$, then $f_1<f_0$ when $r=0$ and $k=d$, while $w_0>w_1$ whenever $\theta_1>\theta_2$. Thus the comparison is strict for every nonuniform distribution.
\end{proof}

In particular, we recover a main result of {\cite{BCF}}. 

\begin{corollary}[{\cite[Theorem 3]{BCF}}]\label{cor:BCF}
With $\theta,\theta_0$ and $T$ as in Corollary \ref{cor:family}, let $M(T)$ be the number of distinct rearrangements of $T$. Then
\[
\E_\theta M(T)\leq\E_{\theta_0}M(T),
\]
with equality if and only if $\theta=\theta_0$.
\end{corollary}

\begin{proof}
Take $\p(x)=\min(1,\tfrac{d}2(1-x))$ and  $g(x)=1/(x-1)$. Since $D_n(\pi)=d$ for exactly $\prod_{j=1}^q n_j!$ permutations and $D_n(\pi)=d-1$ is impossible, we have
\[
A_\p(n)=1-\frac{\prod_{j=1}^q n_j!}{d!}
\quad\text{and}\quad
g(A_\p(n))=-\frac{d!}{\prod_{j=1}^q n_j!}.
\]
Thus $g(A_\p(n(T)))=-M(T)$, and Corollary \ref{cor:family} gives the assertion and its equality case.
\end{proof}

\section{Proof of Theorem \ref{thm:main}}

\begin{proof}
By definition, the theorem is equivalent to
\[
\#\{\pi\in S_d:D_n(\pi)\leq k\}
-\#\{\pi\in S_d:D_m(\pi)\leq k\}\geq0
\]
for every integer $k$. We begin with an identity for this difference. It suffices to consider the case
\begin{equation}\label{eq:move}
m=(r_1+1,r_2,r_3,\ldots,r_q)
\quad\text{and}\quad
n=(r_1,r_2+1,r_3,\ldots,r_q),
\end{equation}
where $r_1>r_2\geq 0$ and both vectors have multiplicities adding up to $d$. Put $r=(r_1,r_2,\ldots,r_q)$ and let $s=(1^{{r_1}},2^{{r_2}},3^{r_3},\ldots,q^{r_q})$ be the canonical word with multiplicities $r$. For $\eta\in S_{d-1}$ and $1\leq\ell\leq q$, define
\[
D_r(\eta)=\#\{1\leq i\leq {d-1}:s_i=s_{\eta(i)}\}
\quad\text{and}\quad
D_{r,\ell}(\eta)=\#\{1\leq i\leq {d-1}:s_i=s_{\eta(i)}=\ell\},
\]
so that $D_r(\eta) = \sum_{\ell =1}^q D_{r,\ell}(\eta)$.
The canonical words $t(m)$ and $t(n)$ differ only at position ${r_1}+1$, which carries $1$ in $t(m)$ and $2$ in $t(n)$. Deleting this position from either word gives $s$. Thus we insert a letter $\ell\in\{1,2\}$ at position ${r_1}+1$ in $s$, obtaining $t(m)$ for $\ell=1$ and $t(n)$ for $\ell=2$. Fix $\eta\in S_{d-1}$, and shift its position labels and their images to the right after $r_1$. For each of the $d-1$ old positions, we obtain an extension by sending that position to ${r_1}+1$ and sending ${r_1}+1$ to its former image, while leaving every other image unchanged. One further extension fixes ${r_1}+1$ and preserves the old bijection. This gives $d$ extensions of $\eta$. Explicitly, letting $\iota(j)=j+{\bf 1}_{j>{r_1}}$, define $\eta_{i_0}\in S_d$ by
\[
\eta_{i_0}(\iota(j))
=\begin{cases}
{r_1}+1,&j=i_0,\\
\iota(\eta(j)),&j\ne i_0,
\end{cases}
\ (j\leq d-1),
\quad 
\text{and} \quad \eta_{i_0}({r_1}+1) =\begin{cases}
\iota(\eta(i_0)), &i_0<d,\\
{r_1}+1,&i_0=d,
\end{cases}
\]
for $i_0=1,\ldots,d$.
Thus $i_0<d$ selects the old position $i_0$, now labelled $\iota(i_0)$, and $i_0=d$ gives the extension fixing ${r_1}+1$. Every permutation in $S_d$ arises like this exactly once. In particular, we derive the identity
\begin{align}
&\#\{\pi\in S_d:D_n(\pi)\leq k\}
-\#\{\pi\in S_d:D_m(\pi)\leq k\} \\
& = \sum_{\pi\in S_{d}}
(\1_{\{D_n(\pi)\leq k\}}
-\1_{\{D_m(\pi)\leq k\}}) \\
& = \sum_{k_0=k}^{k+1}
\sum_{\substack{\eta\in S_{d-1}\\D_r(\eta)=k_0}}
\sum_{i_0=1}^d
(\1_{\{D_n(\eta_{i_0})\leq k\}}
-\1_{\{D_m(\eta_{i_0})\leq k\}}),
\end{align}
where the sum only has terms for $k_0\in \{k,k+1\}$ because if $D_r(\eta)\leq k-1$, both indicators equal \(1\) for every extension in the counts above for $n$ and $m$, while if $D_r(\eta)\geq k+2$, all extensions satisfy $D_n(\eta_{i_0})\geq  k+1$ and $D_m(\eta_{i_0})\geq  k+1$. 

{For the inserted letter $\ell \in \{1,2\}$, the extension associated with position $i=i_0<d$ removes the comparison between $s_i$ and $s_{\eta(i)}$ and compares each of these letters with $\ell$ instead. Exactly $r_\ell$ choices of $i$ have $s_i=\ell$, and exactly $r_\ell$ have $s_{\eta(i)}=\ell$, since $\eta$ is a permutation. Their intersection consists of the $D_{r,\ell}(\eta)$ choices with both letters equal to $\ell$. Thus the old positions split into the following five cases:
\[
\begin{array}{lcc}
\text{Letters }s_i,s_{\eta(i)}&\text{Number of positions }i&\text{Change in matches}\\ \hline
(\ell,\ell)&D_{r,\ell}(\eta)&+1\\
(\ell,\text{a different letter})&r_\ell-D_{r,\ell}(\eta)&+1\\
(\text{a different letter},\ell)&r_\ell-D_{r,\ell}(\eta)&+1\\
\text{equal letters, neither }\ell&D_r(\eta)-D_{r,\ell}(\eta)&-1\\
\text{different letters, neither }\ell&\text{remaining positions}&0
\end{array}
\]
In the first case one match becomes two, in the next two cases one match is created and in the fourth case one match becomes none.}

The extension fixing the new position contributes one more match, so the number of extensions with change $+1$ is
\[
D_{r,\ell}(\eta)+2\bigl(r_\ell-D_{r,\ell}(\eta)\bigr)+1
=1+2r_\ell-D_{r,\ell}(\eta).
\]
The number with change $-1$ is $D_r(\eta)-D_{r,\ell}(\eta)$. Subtracting these two counts from the total of $d$ extensions gives the number with change $0$
\[
d-\bigl(1+2r_\ell-D_{r,\ell}(\eta)\bigr)
-\bigl(D_r(\eta)-D_{r,\ell}(\eta)\bigr)=d-1-D_r(\eta)-2r_\ell+2D_{r,\ell}(\eta).
\]
Finally, if $D_r(\eta)=k$, all extensions except those with change $+1$ give at most $k$ matches. Subtracting the count for $\ell=1$ from that for $\ell=2$ gives $2(r_1-r_2)-D_{r,1}(\eta)+D_{r,2}(\eta)$. If $D_r(\eta)=k+1$, only extensions with change $-1$ qualify, giving the difference $D_{r,1}(\eta)-D_{r,2}(\eta)$. This proves
\begin{equation*}
\begin{aligned}
\#\{\pi\in S_d:D_n(\pi)\leq k\}
-\#\{\pi\in S_d:D_m(\pi)\leq k\}
& =\sum_{\substack{\eta\in S_{d-1}\\D_r(\eta)=k}}
\bigl(2(r_1-r_2)-D_{r,1}(\eta)+D_{r,2}(\eta)\bigr)\\
&\quad\phantom{=}{}+\sum_{\substack{\eta\in S_{d-1}\\D_r(\eta)=k+1}}
\bigl(D_{r,1}(\eta)-D_{r,2}(\eta)\bigr).
\end{aligned}
\end{equation*}
It therefore suffices to prove that, for every integer $k$,
\begin{equation}\label{eq:level-bounds}
0\leq\sum_{D_r(\eta)=k}\bigl({D_{r,1}(\eta)-D_{r,2}(\eta)}\bigr)
\leq{(r_1-r_2)}\,\#\{\eta\in S_{d-1}:D_r(\eta)=k\}.
\end{equation}

To estimate these sums, for $z\in\z_{\geq0}^q$ write $|z|=\sum_{\ell=1}^qz_\ell$ and define
\[
B(z)=
\sum_{\substack{C=(c_{ij})\in\z_{\geq0}^{q\times q},\ c_{ii}=0\\
\sum_jc_{ij}=\sum_jc_{ji}=z_i\text{ for all }i}}
\frac1{\prod_{i\ne j}c_{ij}!}.
\]
Set $B(z)=0$ if any coordinate is negative, and use the convention $1/t!=0$ for negative integers $t$. The function $B$ is symmetric and satisfies $B(x)\geq B(y)$ whenever $x\maj y$. Indeed, $(\prod_{\ell=1}^qz_\ell!)B(z)$ counts the distinct rearrangements of a word with multiplicities $z$ that differ from it at every position. This count increases under balancing by \cite[Theorem 4.3]{Bang}, while the factorial product decreases.

The table $c_{ij}=\#\{v:s_v=i,\ s_{\eta(v)}=j\}$ is realized by $\prod_{\ell=1}^q(r_\ell!)^2/\prod_{i,j}c_{ij}!$ permutations. {We group these permutations by the number of matches of each letter. For a fixed vector $y\in\z_{\geq0}^q$, prescribing $D_{r,\ell}(\eta)=y_\ell$ fixes the diagonal entries $c_{\ell\ell}=y_\ell$. Removing the diagonal leaves row and column sums $r-y$, so summing over the remaining entries gives $B(r-y)$. Hence
\[
\#\{\eta\in S_{d-1}:D_{r,\ell}(\eta)=y_\ell\text{ for all }\ell\}
=\frac{\prod_{\ell=1}^q(r_\ell!)^2}{\prod_{\ell=1}^qy_\ell!}\,B(r-y).
\]
The condition $D_r(\eta)=k$ becomes $|y|=k$, and the weight $D_{r,1}(\eta)-D_{r,2}(\eta)$ becomes $y_1-y_2$. Multiplying the count by this weight and summing over these diagonals gives}
\[
{\displaystyle\sum_{D_r(\eta)=k}\bigl({D_{r,1}(\eta)-D_{r,2}(\eta)}\bigr)}
=
{\prod_{\ell=1}^q(r_\ell!)^2}\sum_{\substack{y\in\z_{\geq0}^q\\|y|=k}}
\frac{(y_1-y_2)B(r-y)}{\prod_{\ell=1}^qy_\ell!}.
\]
Pair $(y_1,y_2)$ with $(y_2,y_1)$. If $y_1>y_2$, then
\[
B({r_1}-y_1,{r_2}-y_2,r_3-y_3,\ldots,r_q-y_q)
\geq
B({r_1}-y_2,{r_2}-y_1,r_3-y_3,\ldots,r_q-y_q),
\]
since the first pair is more balanced.  Thus every paired contribution is nonnegative, proving the first inequality in \eqref{eq:level-bounds}.

{For the second inequality, use the same count and set $z=r-y$. Thus $z_\ell$ counts the positions carrying letter $\ell$ that do not match their images. When $|y|=k$, we have
\[
|z|=d-1-k
\quad\text{and}\quad
r_1-r_2-y_1+y_2=z_1-z_2.
\]
The weight is now $z_1-z_2$, while $B(r-y)$ becomes $B(z)$ and $y_\ell!$ becomes $(r_\ell-z_\ell)!$. Substituting in the same count gives}
\[
{\displaystyle\sum_{D_r(\eta)=k}
\bigl(r_1-r_2-{D_{r,1}(\eta)}+{D_{r,2}(\eta)}\bigr)}
=
{\prod_{\ell=1}^q(r_\ell!)^2}\sum_{\substack{z\in\z_{\geq0}^q\\|z|={d-1}-k}}
\frac{(z_1-z_2)B(z)}{\prod_{\ell=1}^q(r_\ell-z_\ell)!}.
\]
Pair $(z_1,z_2)$ with $(z_2,z_1)$. For integers $z_1>z_2\geq0$,
\[
\frac1{({r_1}-z_1)!({r_2}-z_2)!}\geq\frac1{({r_1}-z_2)!({r_2}-z_1)!}.
\]
If the right side is positive, the ratio of the two sides is
$\prod_{\nu=z_2}^{z_1-1}({r_1}-\nu)/({r_2}-\nu)\geq1$; otherwise the inequality is immediate. Symmetry of $B$ now proves the second inequality in \eqref{eq:level-bounds}.
\end{proof}

\section*{AI Statement}
The author had a proof of Theorem~\ref{thm:main} with no related applications.
It arose from investigations related to \cite{GMR}, which are, in principle, unrelated to \cite{BCF} and its predecessor paper. A closer reading of
\cite[Theorem~3]{BCF}, however, suggested that the results were, in fact, related.
The author then interacted with Generative AI to find a relation, which, in joint effort, led to the corollaries in this paper and simplified parts of the original argument. All proofs were checked by the author, who takes full responsibility for the related content.

\section*{Acknowledgments}
F.G. acknowledges support from the following funding agencies: The Office of Naval Research GRANT14201749, The Serrapilheira Institute (Serra-2211-41824), FAPERJ (E-26/200.209/2023 and E-26/210.245/2024) and CNPq (309910/2023-4).

\end{document}